\documentclass[11pt,a4paper]{article}

\usepackage[numbers,sort&compress]{natbib}
\usepackage{amsmath,amssymb,amsthm}
\usepackage[numbers,sort&compress]{natbib}
\usepackage{graphicx}
\usepackage{hyperref} 

\newcommand{\vol}{\operatorname{vol}}

\newcommand{\diverg}{\operatorname{div}}

\newcommand{\dist}{\operatorname{dist}}

\newcommand{\sgn}{\operatorname{sgn}}

\newtheorem{definition}{Definition}

\newtheorem{theorem}{Theorem}
\newtheorem{remark}{Remark}
\newtheorem{example}{Example}

\begin{document}

\title{\textbf{Weighted phase volume stability: dissipativity and geometric interpretation}}

\author{Igor Furtat\thanks{IPME RAS, Boljshoy prospekt V.O., d. 61,
Saint Petersburg, 199178, Russia. Email: cainenash@mail.ru.
ORCID: 0000-0003-4679-5884}}

\date{}

\maketitle

\begin{abstract}
The evolution of weighted phase volume under a fixed dynamical system is investigated using a positive weight raised to an arbitrary real exponent. Sufficient conditions for uniform exponential contraction and expansion of transported weighted volume are obtained in terms of the corresponding weighted divergence. These conditions make it possible to reveal dissipative properties that may remain undetected by the ordinary divergence. Consequences for invariant sets are established, and, under the assumption of a compact absorbing set, the weighted measure of a global attractor is characterized.
The approach is extended to invariant submanifolds through a distance-weighted ambient-volume estimate. A differential-geometric interpretation is also provided, in which the weighted divergence is identified with the ordinary divergence associated with a conformally transformed metric. The dependence on the exponent is used to reveal a duality between mutually inverse weights. In addition, sufficient local conditions for exponential attraction to an invariant hypersurface are derived using a sign-changing weight. Weighted-volume decay is interpreted as an integral property of transported sets and, by itself, is not identified with pointwise Lyapunov stability.
\end{abstract}


\section{Introduction}

The classical theory of motion stability, founded by A.\,M.\,Lyapunov \cite{Lyapunov}, studies the behavior of individual trajectories and relies on the existence of a Lyapunov function that decreases along solutions. An alternative approach, originating from the results of Liouville \cite{Liouville} and developed in ergodic theory \cite{Sinai,KatokHasselblatt}, considers the evolution of the phase volume (Lebesgue measure) of regions in the phase space. The sign of the divergence of a vector field determines the local contraction or expansion of the ordinary phase volume, and the condition $\diverg f\le 0$ for the system $\dot{x}=f(x)$ ensures nonincrease of transported volume. Existence of a compact global attractor requires additional absorbing and compactness hypotheses \cite{Temam,ChepyzhovVishik,Ladyzhenskaya}. 
Important contributions to the theory of dissipative systems were made in \cite{Levinson}, where the concept of dissipativity via absorbing sets was formulated, and in \cite{NemytskiiStepanov}, where the qualitative theory of differential equations was developed.
In parallel, integral Lyapunov functionals \cite{Krasovskii,Zubov} were developed, allowing stability to be assessed by the behavior of integral characteristics.

The idea of using divergence to analyze the stability of dynamical systems has a rich history, with contributions from many researchers. 
Early results \cite{Zaremba54,Fronteau65,Brauchli68} laid the foundations for studying stability through the sign of the divergence of the vector field, particularly in the context of autonomous systems and index theory. 
These pioneering studies were followed in \cite{Jukov78,Jukov79,Jukov90,Jukov99,Shestakov78}, which developed divergence conditions for stability and instability of nonlinear systems, including criteria based on the behavior of the divergence in invariant domains. 
In parallel, the monograph \cite{Krasnoselski63} provided a systematic treatment of vector fields on the plane, offering topological and analytical tools that complemented the divergence-based approach. 
In \cite{Rantzer00,Rantzer01}, an integral stability criterion dual to Lyapunov's theorem is proposed. Instead of requiring a pointwise decreasing function, Rantzer's approach uses the sign of the divergence to construct a density function that guarantees stability almost everywhere, thereby providing a powerful alternative for systems where classical Lyapunov functions are difficult to find. 
More recently, in \cite{Furtat20a,Furtat20,Furtat21,Furtat22,Furtat26}, the divergence approach was further developed for various classes of dynamical systems, including nonautonomous and perturbed ones. New necessary and sufficient conditions for stability, asymptotic stability, and exponential stability were formulated, together with ellipsoidal reachable set estimates.
The classical foundations of nonlinear systems theory and geometric methods, which underpin many of these developments, are presented in \cite{Arnold,Mishchenko,Khalil,Sontag}.

However, in many physical problems it is natural to consider the weighted phase volume $\int_D w(x)\,dV(x)$, where the positive weight $w$ specifies a reference measure (not necessarily an invariant one) or a conformal change of metric. Such constructions arise in statistical mechanics (canonical Gibbs measures), in Hamiltonian systems with friction is Nosé--Hoover thermostats \cite{Nose,Hoover,EvansMorriss}, and also in geometric analysis when studying flows conformally equivalent to gradient ones \cite{Perelman,Topping}. The weighted measure allows one to account for the nonuniformity of the contraction of the phase flow and often provides more precise information about the dissipative properties of the system.

Recently, a method based on weighted phase volumes was proposed in \cite{Furtat26_Arc}, where the notions of forward and backward weighted divergences were introduced, integral stability criteria were obtained, and ellipsoidal reachable sets were constructed. However, that paper considered only the special cases $\alpha=1$ and $\alpha=-1$, which correspond to the weight $w$ and its inverse, respectively. The present paper takes the next natural step: we introduce a one-parameter family of weighted volumes $\Omega_{w^\alpha}$ with weight $w^\alpha$, parameterized by an arbitrary real exponent $\alpha\in\mathbb{R}$. This recovers the classical case $\alpha=0$ and connects the weighted divergences for positive and negative exponents. These signs do not themselves imply contraction or expansion. The main tool becomes the \emph{weighted divergence of order $\alpha$}, which will be defined in the paper.

The approach provides integral estimates for transported sets near invariant sets (including limit cycles and submanifolds). Such estimates alone do not establish pointwise Lyapunov stability of individual trajectories. In particular, the weighted volume can serve as an integral Lyapunov functional. The choice of the parameter $\alpha$ and the function $w$ allows one to flexibly adjust its properties to the geometry of a particular flow. Moreover, the developed approaches are related to conformal geometry, i.e., the weighted divergence of order $\alpha$ turns out to be the ordinary divergence in a conformally transformed metric, and the associated rescalings form a one-parameter group for fixed $w$. This supplies a Riemannian interpretation of the volume identities.

The present paper is also connected to several recent developments in the theory of almost global convergence and dissipativity. In \cite{Kamaldar24}, Lyapunov-like functions are used for almost global convergence in discrete-time systems. The present volumetric estimates address transported measures, a different conclusion from almost-global pointwise convergence. 
The weighting in multiobjective strict dissipativity in \cite{Grune22} serves a different purpose. Here $w^\alpha$ changes the reference phase-volume measure, with affine dependence of its divergence on $\alpha$. 
Contraction analysis of functional differential equations with switching is studied in \cite{Tran24}. Our setting instead concerns the evolution of weighted ambient volumes of sets under finite-dimensional flows. 
The paper \cite{Lorenzetti25} investigates almost global attractivity via semigroup theory. That work concerns almost-global attraction, whereas our principal conclusions are set-wise volume estimates rather than pointwise attraction.
Thus, the presented framework unifies these directions by introducing a one-parameter family of weighted volumes and weighted divergences, providing explicit exponential weighted-volume estimates and a coherent geometric interpretation.

The main novelty of this paper is as follows:
\begin{itemize}
\item[(i)] A one-parameter family of weighted volumes with weight $w^\alpha$, $\alpha\in\mathbb{R}$, and the corresponding divergences $D_{w^\alpha}f$ is considered, which allows one to test both contraction and expansion for different parameter choices. The sign of $\alpha$ alone does not determine either behavior. The cases $\alpha=1$ and $\alpha=-1$, studied in \cite{Furtat26_Arc}, become particular instances.
\item[(ii)] Sufficient conditions for $w^\alpha$-volumetric dissipativity and contractivity are given, including exponential estimates, absence of invariant sets of positive measure, and a description of the global attractor for any $\alpha$. In contrast to classical divergence criteria \cite{Jukov78,Jukov90,Rantzer01}, the proposed conditions do not require sign-definiteness of the ordinary divergence and are applicable to a wider class of systems.
\item[(iii)] Generalizations to the case of attraction to a submanifold $N$ of arbitrary codimension are obtained, and sufficient conditions for exponential $N$-$w^\alpha$-contraction are derived. This quantifies the evolution of ambient volume weighted by distance from the submanifold. Pointwise transverse stability requires additional hypotheses.
\item[(iv)] The geometric nature of the weighted divergence as the ordinary divergence in a conformally transformed metric is investigated, which realizes the family as a one-parameter family of conformal rescalings of the fixed background metric.
\item[(v)] The duality $\alpha\leftrightarrow -\alpha$ is shown and the symmetry between contraction and expansion of volumes is investigated, as an algebraic identity for weighted divergences.
\item[(vi)] An oriented volumetric stability for sign-alternating weights arising in the consideration of invariant hypersurfaces is developed. Sufficient conditions for asymptotic stability are obtained, generalizing classical results on dissipativity.
\end{itemize}

The paper is organized as follows. Section~\ref{Sec:PrSt} formulates the problem statement and introduces the main definitions. Section~\ref{sec:general} contains results on $w^\alpha$-volumetric stability: uniform rates of contraction and expansion, invariant sets of positive measure, a set-functional estimate, existence of a global attractor, and the nonautonomous case. In Section~\ref{sec:wsubmanifold_alpha}, the developed approaches are extended to the case of attraction to a submanifold with weight $w^\alpha$. Section~\ref{sec:wforms} is devoted to the differential-geometric interpretation and conformal equivalence for arbitrary $\alpha$. Section~\ref{sec:orient} treats local signed-weight stability near an invariant hypersurface. The algebraic duality $\alpha\leftrightarrow-\alpha$ is discussed in Section~\ref{sec:wforms}. 

Throughout this paper, we use the following \textit{notation}:
$\mathbb{R}$ denotes the set of real numbers;
$\mathbb{N}$ denotes the set of natural numbers;
$\varnothing$ denotes the empty set;
for a set $A\subset M$, $\overline{A}$ denotes its closure in $M$;
for a set $A\subset M$, $\operatorname{int}(A)$ denotes its interior;
$\operatorname{vol}(A)$ denotes the Riemannian volume of a measurable set $A\subset M$;
$\varphi_t\colon M\to M$ denotes the flow generated by the vector field $f$;
$\diverg f$ denotes the classical Riemannian divergence of the vector field $f$;
$\nabla$ denotes the gradient operator on $M$;
$\mathcal{L}_f$ denotes the Lie derivative along $f$ (for a scalar function $h$, $\mathcal{L}_f h=dh(f)=\langle\nabla_g h,f\rangle_g$); throughout, the notation $f^\top\nabla h$ is shorthand for $dh(f)$ even in non-Cartesian coordinates;
$\wedge$ denotes the exterior (wedge) product of differential forms.

\section{Problem statement}
\label{Sec:PrSt}

Consider an autonomous dynamical system of the form
\begin{equation}
\label{eq:sys}
\dot{x}=f(x),~~ x\in M,
\end{equation}
where $M$ is an $n$-dimensional smooth oriented Riemannian manifold with volume form $dV$, and $f$ is a smooth vector field on $M$ generating a forward-complete flow $\varphi_t\colon M\to M$ for $t\ge0$.

Let a positive function $w\in C^1(M)$, $w>0$, and a real parameter $\alpha\in\mathbb{R}$ be given. For a relatively compact measurable set $D\subset M$ (i.e. with compact closure in $M$), we define the \emph{weighted phase volume of order $\alpha$} \cite{Furtat26_Arc} as
\begin{equation}
\label{eq:Vwalpha_ps}
\Omega_{w^\alpha}(D)=\int_D w(x)^\alpha\, dV(x).
\end{equation}
The evolution of the set under the flow is given by $D(t)=\varphi_t(D)$. For every such measurable set (not necessarily a domain with smooth boundary), the change-of-variables formula for the flow and differentiation under the integral on compact flow tubes give
\begin{equation}
\begin{array}{lll}
\label{eq:deriv_alpha}
\frac{d}{dt}\Omega_{w^\alpha}(D(t)) = \int_{D(t)} D_{w^\alpha}f(x)\, w(x)^\alpha\, dV(x),
\end{array}
\end{equation}
where we have introduced the \emph{weighted divergence of order $\alpha$} \cite{Furtat26_Arc} defined by
\begin{equation}
\label{eq:Dwalpha}
D_{w^\alpha}f(x) := \diverg f(x) + \alpha \mathcal{L}_f\ln w(x).
\end{equation}

\begin{remark}
For $\alpha=0$, the weighted divergence reduces to the classical divergence $\diverg f$ \cite{Arnold}. 
Hence, classical results on volumetric dissipativity and contractivity based on the sign of $\diverg f$ are recovered as special cases of the present framework.
\end{remark}

The goal of this paper is to obtain conditions, expressed in terms of the weighted divergence $D_{w^\alpha}f$, under which the weighted volume $\Omega_{w^\alpha}(D(t))$ is monotonically decreasing (dissipativity) or exponentially decreasing (contractivity), and to investigate the resulting properties of the dynamical system. In particular, we study for which $\alpha$ and $w$ the system is $w^\alpha$-volumetrically dissipative, when exponential contraction or expansion of the volume occurs with explicit estimates, what invariant sets may exist in dissipative systems with $D_{w^\alpha}f\le0$, how global attractors are structured in $w^\alpha$-contracting systems and what their $w^\alpha$-measure is, how weighted-volume estimates depend on the parameter $\alpha$ and whether there is a duality between $\alpha$ and $-\alpha$, what the geometric meaning of $D_{w^\alpha}f$ is, and whether the developed approach can be extended to systems with attraction to a submanifold and to the case of a sign-alternating weight generated by an invariant hypersurface.

From the perspective of stability theory, the developed approach provides a method for constructing integral Lyapunov functionals in the form of the weighted volume $\Omega_{w^\alpha}$, whose monotonicity or exponential decay controls an integral of transported domains. Convergence of individual trajectories to an attractor requires separate arguments.

\section{$w^\alpha$-volumetric stability}
\label{sec:general}

In this section, the one-parameter family $\Omega_{w^\alpha}$ includes ordinary phase volume ($\alpha=0$) and the measures corresponding to positive and negative exponents. The sign of $\alpha$ alone does not fix the sign of the weighted divergence. The main focus is on how the choice of the parameter $\alpha$ and the function $w$ leads to explicit exponential estimates of contraction or expansion of the volume, thus providing a quantitative measure of dissipation. Moreover, the role of the weighted volume itself as an integral Lyapunov functional, adjustable to the geometry of the flow, is discussed. In other words, we show that by choosing a suitable ``weight'' (density of the measure) and the parameter $\alpha$, one can modify the way the phase volume is measured so that its value decreases monotonically along the flow, even when the ordinary (uniform) phase volume does not possess this property. This approach provides estimates for transported domains near invariant sets. It does not alone prove pointwise stability.

Throughout this section, unless otherwise stated, domains have compact closure in $M$ and the flow is forward-complete. The transport formula holds for measurable such domains. Jordan measurability is invoked only where needed. We first introduce the definitions used in this section.

\begin{definition}
\label{Def:def1}
System \eqref{eq:sys} is called $w^\alpha$-volumetrically conservative if for any measurable domain $D \Subset M$ and all $t\ge0$ one has
$\Omega_{w^\alpha}(\varphi_t(D))=\Omega_{w^\alpha}(D)$ (for continuous $D_{w^\alpha}f$ this is equivalent to $D_{w^\alpha}f\equiv 0$ on $M$).
\end{definition}

\begin{definition}
\label{Def:def2}
System \eqref{eq:sys} is $w^\alpha$-volumetrically dissipative if for any measurable domain $D\Subset M$ and all $t\ge0$ one has
$\Omega_{w^\alpha}(\varphi_t(D))\le \Omega_{w^\alpha}(D)$ 
(for continuous $D_{w^\alpha}f$ this is equivalent to the inequality $D_{w^\alpha}f(x)\le 0$ for all $x\in M$).
If the condition is restricted to domains within an open positively invariant region $U$, we speak of regional $w^\alpha$-volumetric dissipativity.
\end{definition}

\begin{definition}
\label{Def:def3}
System \eqref{eq:sys} is called asymptotically $w^\alpha$-volumetrically contracting if, for every measurable $D\Subset M$,
$\lim_{t\to\infty}\Omega_{w^\alpha}(\varphi_t(D))=0$.
\end{definition}

Suppose $U$ is a positively invariant region on which $0<c\le w^\alpha(x)\le C<\infty$. For every measurable $D\Subset U$ and all $t\ge0$, the entire transported set $\varphi_t(D)$ remains in $U$, and therefore $c\,\vol(\varphi_t(D))\le\Omega_{w^\alpha}(\varphi_t(D))\le C\,\vol(\varphi_t(D))$. Under these hypotheses the weighted and ordinary volumes tend to zero simultaneously. Their instantaneous monotonicity in time need not be equivalent. Different exponents may change the contraction or expansion of the measured volume, but do not change the underlying trajectories. The choice of $w$ and $\alpha$ can be interpreted as a transition to a new Riemannian metric $\tilde g_\alpha = w^{2\alpha/n} g$, conformally equivalent to the original one, in which the ordinary divergence of the field $f$ equals $D_{w^\alpha}f$ (see Section~\ref{sec:wforms}). Thus, $w^\alpha$-dissipativity is equivalent to ordinary dissipativity in the conformally transformed geometry with metric $\tilde g_\alpha$.

We now state the result on the uniform rate of $w^\alpha$-contraction.

\begin{theorem}
\label{th:wrate_alpha_comp}
Let $U\subseteq M$ be an open positively invariant domain, $\lambda>0$, and suppose
$D_{w^\alpha}f(x) \le -\lambda$ for all $x\in U$.
Then for every measurable $D\Subset U$ the estimate
\begin{equation}
\label{ineq:dissipat}
\Omega_{w^\alpha}(\varphi_t(D)) \le e^{-\lambda t}\,\Omega_{w^\alpha}(D),\qquad t\ge0,
\end{equation}
holds. If $U=M$, then the system is asymptotically $w^\alpha$-contracting at the stated volume rate for all $D\Subset M$.
\end{theorem}

\begin{remark}
Differently from the notion of dissipativity in the sense of Levinson \cite{Levinson}, which only requires the existence of absorbing sets, the proposed approach provides a quantitative measure of dissipation. 
Indeed, if the condition $D_{w^\alpha}f(x) \le -\lambda < 0$ holds for all $x$ in a positively invariant domain $U\subset M$, then by Theorem~\ref{th:wrate_alpha_comp} the weighted volume of any initial domain $D\subset U$ satisfies the explicit exponential estimate \eqref{ineq:dissipat}. 
This is a quantitative estimate for phase volume, distinct from the existence of an absorbing set. It does not by itself estimate convergence times of individual trajectories.
\end{remark}

\begin{proof}
Consider an initial measurable domain $D\Subset U$. The flow $\varphi_t$ is forward-complete. Define the function of time
\begin{equation}
\label{eq:V_D}
V(t): = \Omega_{w^\alpha}(D(t)) = \int_{D(t)} w(x)^\alpha \, dV(x),
\end{equation}
where $D(t) = \varphi_t(D)$. At $t=0$, the value of $V(0)$ is finite since $\overline D$ is compact and $w^\alpha$ is continuous and positive there, where here and throughout $\overline{D}$ denotes the closure of the domain $D$.
Using \eqref{eq:Vwalpha_ps}--\eqref{eq:Dwalpha} and the assumptions of the theorem, there exists a constant $\lambda > 0$ such that
$D_{w^\alpha}f(x) \le -\lambda$ for all $x \in U$.
The domain $U$ is positively invariant, i.e. $D(t) \subset U$ for all $t \ge 0$. Hence, for every point $x \in D(t)$ we have $D_{w^\alpha}f(x) \le -\lambda$. Multiplying this inequality by the positive function $w^\alpha$, we obtain the pointwise estimate of the integrand:
\begin{equation*}
D_{w^\alpha}f(x) \, w(x)^\alpha \le -\lambda \, w(x)^\alpha, ~~ x \in D(t) .
\end{equation*}

Integrating this inequality over $D(t)$ and using \eqref{eq:Vwalpha_ps}, \eqref{eq:deriv_alpha} and \eqref{eq:V_D}, we obtain
$\dot V(t) \le -\lambda \int_{D(t)} w^\alpha \, dV = -\lambda V(t)$.
Thus, the function $V(t)$ satisfies the differential inequality
$\dot V(t) + \lambda V(t) \le 0$, $t \ge 0$.
Multiplying by the integrating factor $e^{\lambda t}$, we get
$\frac{d}{dt}\bigl( V(t) e^{\lambda t} \bigr)
= e^{\lambda t} \bigl( \dot V(t) + \lambda V(t) \bigr) \le 0$.
Consequently, the product $V(t) e^{\lambda t}$ is nonincreasing in $t$. In particular, for any $t \ge 0$ we have
$V(t) e^{\lambda t} \le V(0) e^{0} = V(0)$.
Dividing by $e^{\lambda t} > 0$, we arrive at the required exponential estimate
\begin{equation*}
\Omega_{w^\alpha}(\varphi_t(D)) = V(t) \le V(0) e^{-\lambda t} = \Omega_{w^\alpha}(D) e^{-\lambda t}.
\end{equation*}

If $U=M$, the same estimate holds for each $D\Subset M$ and all $t\ge0$. Passing to the limit gives
$\lim_{t\to\infty} \Omega_{w^\alpha}(\varphi_t(D)) = 0$.
This means that the system is asymptotically $w^\alpha$-contracting. Theorem \ref{th:wrate_alpha_comp} is proved.
\end{proof}

We now state the result on the uniform rate of $w^\alpha$-expansion.

\begin{theorem}
\label{th:wrate_alpha_exp}
If in an open positively invariant domain $U$ we have $D_{w^\alpha}f(x) \ge \mu$ for some $\mu>0$, then for every measurable $D\Subset U$ one gets 
$\Omega_{w^\alpha}(\varphi_t(D)) \ge e^{\mu t}\,\Omega_{w^\alpha}(D)$, $t\ge0$.
\end{theorem}
\begin{proof}
Analogous to the proof of Theorem \ref{th:wrate_alpha_comp}, using \eqref{eq:deriv_alpha} and $D_{w^\alpha}f\ge \mu$ we obtain $\dot V(t) \ge \mu V(t)$, which yields the lower estimate.
\end{proof}

\begin{example}
\label{ex:3.4_3.5}
Consider the forward-complete system on $M=\mathbb R^2\setminus\{0\}$
\begin{equation}\label{ex:sys_first}
\dot r=r(1-r),\qquad \dot\theta=1,
\end{equation}
which has an attracting limit cycle at $r=1$. The Euclidean divergence in the polar coordinate basis is
$\operatorname{div}f=r^{-1}\partial_r(r\dot r)+\partial_\theta\dot\theta=2-3r$. It is positive for $r<2/3$ and negative for $r>2/3$. Thus its sign varies on the annulus $1/2\le r\le3/2$, although it is strictly negative in a sufficiently small neighbourhood of the cycle.

Choose $w(r)=e^{16(r-1)^2}$ and $\alpha=1$. Then
$D_wf=2-3r-32r(r-1)^2$.

The compact annulus $K=\{1/2\le r\le3/2\}$ is positively invariant and contained in the positively invariant open annulus $U=\{0.49<r<1.51\}$. On $U$ and hence on $K$ the maximum of $D_wf$ is approximately $-0.93$, reached at $r\approx0.95$ (the interior critical point of $D_wf$). In particular $D_wf\le-1/2$ on $U$. Hence Theorem~\ref{th:wrate_alpha_comp} yields
$\Omega_w(\varphi_t(D))\le e^{-t/2}\Omega_w(D)$ for every measurable $D\subset K$.
For $\alpha=-1$, the expression $D_{w^{-1}}f=2-3r+32r(r-1)^2$ changes sign on $K$, so neither uniform-sign theorem applies there. The example concerns weighted volume. The pointwise attraction of the cycle follows separately from the radial equation.
\end{example}

We now formulate a result on the absence of $w^\alpha$-invariant sets of positive measure.

\begin{theorem}
\label{th:wwandering_alpha}
Assume that the system is $w^\alpha$-volumetrically dissipative ($D_{w^\alpha}f\le0$ almost everywhere) and, moreover, the set of points where $D_{w^\alpha}f(x)<0$ is dense in $M$.
Then there exist no relatively compact Jordan-measurable invariant (i.e. $\varphi_t(A)=A$ for all $t$) sets $A$ with $\Omega_{w^\alpha}(A)>0$.
\end{theorem}

\begin{remark}
The weighted volume $\Omega_{w^\alpha}(D)$ itself serves as a Lyapunov functional whose properties are completely determined by the sign of $D_{w^\alpha}f$. Thanks to the parameter $\alpha$, this functional can be tuned and adapted to the geometry of the flow without solving additional equations.
\end{remark}

\begin{proof}
Let $A\subset M$ be an arbitrary relatively compact, Jordan-measurable and flow-invariant set (i.e. $\varphi_t(A)=A$ for all $t\ge0$). We show that the assumptions of the theorem imply $\Omega_{w^\alpha}(A)=0$.

By invariance of $A$, the weighted volume $\Omega_{w^\alpha}(\varphi_t(A))$ is constant and equal to $\Omega_{w^\alpha}(A)$. Differentiating this equality at $t=0$ using \eqref{eq:deriv_alpha}, we get
\begin{equation*}
0 = \frac{d}{dt}\bigg|_{t=0} \Omega_{w^\alpha}(\varphi_t(A)) = \int_A D_{w^\alpha}f(x)\, w(x)^\alpha\, dV(x).
\end{equation*}
By the assumption of the theorem, $D_{w^\alpha}f(x)\le 0$ for almost all $x\in M$, and the weight $w^\alpha$ is strictly positive. Hence the integrand $D_{w^\alpha}f \cdot w^\alpha$ is almost everywhere non-positive, and its integral over $A$ is zero. This is possible only if $D_{w^\alpha}f(x)=0$ almost everywhere on $A$. By continuity of $D_{w^\alpha}f$, the equality holds on the open set $U:=\operatorname{int}(A)$: if it failed at an interior point, it would fail on a neighbourhood of positive measure within $A$.

On the other hand, the set $\{x\in M : D_{w^\alpha}f(x)<0\}$ is dense in $M$. If $U$ were nonempty, it would contain at least one point where $D_{w^\alpha}f<0$, contradicting the above. Therefore, $\operatorname{int}(A)=\varnothing$.

For a Jordan-measurable set, the boundary has Lebesgue measure zero, so its volume equals the volume of its interior: $\vol(A)=\vol(\operatorname{int}(A))=0$. Then
\begin{equation*}
\Omega_{w^\alpha}(A) = \int_A w(x)^\alpha\, dV(x)
\le \bigl(\sup_{x\in A} w(x)^\alpha\bigr) \vol(A) = 0.
\end{equation*}
Since $\Omega_{w^\alpha}(A)\ge0$, we get $\Omega_{w^\alpha}(A)=0$. 
Theorem \ref{th:wwandering_alpha} is proved.
\end{proof}

\begin{example}
\label{ex:th2_alpha}
Consider system \eqref{ex:sys_first} on $M=\{(r,\theta):r>1/4\}$ and the smooth positive weight $w(r)=e^{3r^2}/r^3$.
We have $\ln w = 3r^2-3\ln r$, $\nabla\ln w = (6r-3/r) \partial_r$, where $\partial_r$ is the unit radial vector in polar coordinates on the plane. 
Compute $f^{\top} \nabla\ln w = 3(2r^2-1)(1-r)$.
Then
$D_{w^\alpha}f = 2-3r + 3\alpha(2r^2-1)(1-r)$.
For $\alpha=1$ we get $D_w f = -1+6r^2(1-r)$. 
The maximum of $D_{w}f$ is $-1/9<0$ attained at $r=2/3$.
By Theorem~\ref{th:wwandering_alpha}, there exists no relatively compact Jordan-measurable invariant set of positive $w$-measure. For other $\alpha$, the sign of $D_{w^\alpha}f$ may change, which gives additional information about the system.
\end{example}

We now formulate a result on the weighted volumetric Lyapunov functional.

\begin{theorem}
\label{th:wLyap_alpha}
If the system is both $w^\alpha$-volumetrically dissipative and asymptotically $w^\alpha$-contracting, then the set functional
$F_{w^\alpha}(D)=\Omega_{w^\alpha}(D)$ has the following properties:
\begin{enumerate}
\item $F_{w^\alpha}(\varphi_t(D))$ is monotonically nonincreasing and tends to zero as $t\to\infty$;
\item the level sets $\{D: F_{w^\alpha}(D)\le c\}$ are invariant in the sense that from ${F}_{w^\alpha}(D)\le c$ it follows that $F_{w^\alpha}(\varphi_t(D))\le c$ for all $t\ge0$.
\end{enumerate}
\end{theorem}


\begin{proof}
By definition of $w^\alpha$-volumetric dissipativity, for any measurable domain $D\Subset M$ and all $t\ge0$ we have
$\Omega_{w^\alpha}(\varphi_t(D)) \le \Omega_{w^\alpha}(D)$.
This means that the functional $F_{w^\alpha}(D)=\Omega_{w^\alpha}(D)$ is monotonically nonincreasing along the flow.
By definition of asymptotic $w^\alpha$-contractivity,
$\lim_{t\to\infty}\Omega_{w^\alpha}(\varphi_t(D)) = 0$,
i.e. $F_{w^\alpha}(\varphi_t(D))\to 0$ as $t\to\infty$. The first property is proved.

Now suppose $F_{w^\alpha}(D)\le c$, i.e. $\Omega_{w^\alpha}(D)\le c$. By monotonic nonincrease,
$\Omega_{w^\alpha}(\varphi_t(D)) \le \Omega_{w^\alpha}(D) \le c$ for all $t\ge0$.
Hence $\varphi_t(D)$ remains in the level set $\{G : F_{w^\alpha}(G)\le c\}$, proving the second property.
This proves the stated set-functional properties.
\end{proof}

We now state a result on the existence of a global $w^\alpha$-attractor.

\begin{theorem}
\label{th:wattractor_alpha}
Suppose the system is asymptotically $w^\alpha$-contracting and has a compact positively invariant absorbing set $B$.
Thus $\varphi_t(B)\subseteq B$ for $t\ge0$, and each bounded set $D$ is mapped into $B$ after a sufficiently large time. Then
$\mathcal A=\bigcap_{t\ge0}\varphi_t(B)$
is a global attractor with $\Omega_{w^\alpha}(\mathcal A)=0$.
\end{theorem}

\begin{proof}
Consider the family of sets $\{\varphi_t(B)\}_{t\ge0}$. 
By positive invariance of $B$ ($\varphi_t(B) \subseteq B$ for all $t\ge0$) and the semigroup property of the forward flow $\varphi_{t+s}=\varphi_t\circ\varphi_s$, this family is decreasing: for any $t_2\ge t_1\ge0$ we have $\varphi_{t_2}(B) = \varphi_{t_1}(\varphi_{t_2-t_1}(B)) \subseteq \varphi_{t_1}(B)$. 

Define $\mathcal{A} = \bigcap_{t\ge0} \varphi_t(B)$.
As an intersection of nested nonempty compact sets, $\mathcal{A}$ is nonempty and compact. Continuity of the finite-dimensional flow, compactness and the absorbing property imply that $\mathcal{A}$ is invariant and attracts every bounded set, hence is a global attractor (see \cite{Temam}).

By construction $\mathcal{A}\subseteq \varphi_t(B)$ for any $t\ge0$. Monotonicity of the integral with respect to the positive measure $w^\alpha\,dV$ gives
$\Omega_{w^\alpha}(\mathcal{A}) \le \Omega_{w^\alpha}(\varphi_t(B))$ for all $t\ge0$.
Choose a relatively compact open domain $G$ containing the compact set $B$. Asymptotic $w^\alpha$-contractivity gives $\Omega_{w^\alpha}(\varphi_t(G))\to0$, and monotonicity of the measure yields $0\le\Omega_{w^\alpha}(\varphi_t(B))\le\Omega_{w^\alpha}(\varphi_t(G))\to0$.
Passing to the limit $t\to\infty$ in the inequality, we get
$\Omega_{w^\alpha}(\mathcal{A}) \le \lim_{t\to\infty} \Omega_{w^\alpha}(\varphi_t(B)) = 0$.

Since $\Omega_{w^\alpha}(\mathcal{A})$ is an integral of a positive function, it is nonnegative. Hence $\Omega_{w^\alpha}(\mathcal{A}) = 0$. 
Theorem \ref{th:wattractor_alpha} is proved.
\end{proof}

\begin{example}
For system \eqref{ex:sys_first} on $M=\{r>1/4\}$ with $w=e^{3r^2}/r^3$ and $\alpha=1$, the compact annulus $B=\{1/2\le r\le 2\}$ is positively invariant and absorbs every bounded subset of $M$. 
The attractor $\{r=1\}$ has zero $w$-volume. For other $\alpha$, if contractivity is preserved, the attractor will also have zero $w^\alpha$-volume.
\end{example}

We now consider a generalisation to nonautonomous systems.

\begin{theorem}\label{th:wnonaut_alpha}
Consider a $T$-periodic nonautonomous system $\dot x=f(t,x)$, locally Lipschitz in $x$, whose time-$T$ map $P=\varphi_{0,T}$ is a $C^1$ diffeomorphism onto its image. Assume that $P$ admits a compact positively invariant set $B$ absorbing every bounded subset of $M$. Suppose that for some $\kappa\in(0,1)$,
\[
\Omega_{w^\alpha}(P(D))\le\kappa\,\Omega_{w^\alpha}(D)
\]
for every relatively compact measurable domain $D\Subset M$. Then the global attractor $\mathcal A=\bigcap_{k\ge0}P^k(B)$ has zero $w^\alpha$-volume.
\end{theorem}
\begin{proof}
The sets $P^k(B)$ are nested nonempty compact sets. Their intersection $\mathcal A$ is nonempty, compact and invariant. Compactness and the absorbing property imply that it attracts every bounded set. Choose a relatively compact open domain $G\supset B$. Iteration of the assumed inequality gives
\[
\Omega_{w^\alpha}(P^k(G))\le\kappa^k\Omega_{w^\alpha}(G) \to 0.
\]
Since $\mathcal A\subset P^k(B)\subset P^k(G)$ for every $k$, positivity of the measure gives $\Omega_{w^\alpha}(\mathcal A)=0$.
\end{proof}

\begin{example}
Consider the $2$-periodic polar system
\[
\dot r=a(t)r-r^3,\qquad \dot\theta=1,
\]
where $a(t)=\alpha_1>0$ for $t\in[0,1)$ and $a(t)=\alpha_2>0$ for $t\in[1,2)$, extended periodically. The solutions are understood in the Carath\'eodory sense at switching times. Put
\begin{align*}
A&=\alpha_1+\alpha_2>0,\\
C&=2\int_0^2\exp\!\left(2\int_0^s a(u)\,du\right)ds>0.
\end{align*}
Solving the radial equation gives the time-$2$ map
\begin{align*}
P(r,\theta)=(R(r),\theta+2),~~~
R(r)=\frac{r e^A}{\sqrt{1+Cr^2}}.
\end{align*}
The unique positive fixed radius is
$\rho_*^2=(e^{2A}-1)/C$, with $R'(\rho_*)=e^{-2A}<1$.
Choose a sufficiently small $\varepsilon>0$ satisfying
$\varepsilon<\min\{\rho_*/2,\sqrt{\alpha_1},\sqrt{\alpha_2}\}$ and work on the forward-invariant phase space $M=\{r>\varepsilon\}$. The compact annulus
$B=\{\rho_*/2\le r\le3\rho_*/2\}$ is positively invariant for $P$ and absorbs every bounded subset of $M$ under iteration.

For $w(r)=r^{-3}$ and $\alpha=1$, the weighted polar area element is $r^{-2}dr\,d\theta$, and the ratio of weighted image area to weighted initial area at radius $r$ is
\[
g(r)=\frac{r^2 R'(r)}{R(r)^2}
=\frac{e^{-A}}{\sqrt{1+Cr^2}}\le e^{-A}<1.
\]
The change-of-variables formula therefore gives
$\Omega_w(P(D))\le e^{-A}\Omega_w(D)$ for every measurable $D\Subset M$, and, in particular, for $D=B$. The invariant core $\bigcap_{k\ge0}P^k(B)=\{r=\rho_*\}$ is a zero-$w$-area global attractor for the discrete-time system on $M$.

For the inverse weight $w^{-1}=r^3$, the corresponding ratio is
$R(r)^4R'(r)/r^4=e^{5A}(1+Cr^2)^{-7/2}$, which exceeds $1$ near $r=0$ (and for points of $M$ sufficiently close to $\varepsilon$ if $\varepsilon$ is chosen small enough). Thus the same global contraction bound does not hold for this inverse measure. The invariant circle does not become a repeller: only its measured volume behaviour changes.
\end{example}

\begin{remark}
Classical Lyapunov stability \cite{Lyapunov} studies the behavior of individual trajectories and relies on the existence of a Lyapunov function decreasing along solutions. In contrast, the approach developed in this paper introduces an integral criterion in which the weighted volume $\Omega_{w^\alpha}(D)$ serves as a new Lyapunov functional whose monotonicity is determined by the sign of $D_{w^\alpha}f$. This controls weighted volumes of transported domains and provides explicit exponential integral estimates, not automatic pointwise stability. The flexibility in choosing the parameter $\alpha$ and the weight $w$ makes the method applicable to systems for which the construction of a classical pointwise Lyapunov function is difficult, and allows one to detect dissipative properties hidden from standard analysis.
\end{remark}

\section{Stability with respect to a submanifold with weight $w^\alpha$}
\label{sec:wsubmanifold_alpha}

In many problems of stability theory, the attractor is not an isolated equilibrium point but an invariant submanifold of positive codimension, for example, a limit cycle or an invariant torus.  
To quantify the distribution of transported volume relative to such a set, it is convenient to track a distance-weighted ambient-volume moment, whose integrand vanishes on the submanifold and is positive away from it.  
Geometrically, this construction weights ambient volume according to distance from the submanifold. It need not equal the Jacobian of the normal variational flow.  
Taking the distance-weighted ambient element $\rho_N^m\,dV$ (where $\rho_N$ is the distance to the submanifold $N$ in the power of the codimension $m$) and multiplying it by the weight function $w^\alpha$, we obtain the object $\mathcal{V}_{N,w^\alpha}(D)$, which is sensitive both to the proximity to $N$ and to the choice of the conformal measure.  
The evolution of this moment is described by a distance-weighted divergence identity. Suitable choices of $\alpha$ and $w$ can yield exponential decay of the moment for transported sets in a positively invariant tubular neighbourhood, even when the ordinary divergence gives no such estimate. This integral conclusion does not by itself imply pointwise attraction to the submanifold.  
In this section we extend the weighted-volume estimates to distance-weighted moments near invariant submanifolds of arbitrary codimension.

Let $N\subset M$ be a smooth compact embedded submanifold without boundary of codimension $m\ge1$, and let $U$ be a sufficiently small tubular neighbourhood on which $\rho_N(x)=\operatorname{dist}_g(x,N)$ is smooth off $N$. For $\alpha\in\mathbb{R}$ and a positive function $w$, define, for $D\Subset U$,
\begin{equation}
\label{eq:VNw_alpha}
\mathcal{V}_{N,w^\alpha}(D)=\int_D \rho_N(x)^m\, w(x)^\alpha\, dV(x).
\end{equation}

We now modify Definitions~\ref{Def:def1}--\ref{Def:def3} for the case of stability with respect to a submanifold with weight $w^\alpha$.

\begin{definition}
The system is $N$-$w^\alpha$-volumetrically dissipative if there exists a positively invariant tubular neighbourhood $U\supset N$ such that for every measurable $D\Subset U$ and all $t\ge0$, we have $\mathcal{V}_{N,w^\alpha}(\varphi_t(D))\le \mathcal{V}_{N,w^\alpha}(D)$.
\end{definition}

\begin{definition}
The system is asymptotically $N$-$w^\alpha$-contracting if for every such $D\Subset U$ one has $\lim_{t\to\infty}\mathcal{V}_{N,w^\alpha}(\varphi_t(D))=0$.
\end{definition}

We now state the condition for exponential $N$-$w^\alpha$-contraction.

\begin{theorem}
\label{th:wNmain_alpha}
Assume that $N$ is invariant and that there exists a positively invariant tubular neighbourhood $U\supset N$ and a constant $\lambda>0$ such that for $x\in U\setminus N$ the transverse dissipativity condition
\begin{equation}
\label{eq:wcond_alpha}
D_{w^\alpha}f(x) + m\,\frac{f(x)^{\top}\nabla\rho_N(x)}{\rho_N(x)} \le -\lambda
\end{equation}
holds. Then for every measurable initial domain $D\Subset U$,
\begin{equation}
\label{eq:est_exp_N}
\mathcal{V}_{N,w^\alpha}(\varphi_t(D)) \le e^{-\lambda t}\,\mathcal{V}_{N,w^\alpha}(D), ~~ t\ge0.
\end{equation}
\end{theorem}

\begin{proof}
Define the function of time
\begin{equation*}
V(t): = \mathcal{V}_{N,w^\alpha}(D(t)) = \int_{D(t)} \rho_N(x)^m \, w(x)^\alpha \, dV(x).
\end{equation*}
The distance $\rho_N$ is smooth on $U\setminus N$ and locally Lipschitz on $U$. Since $m\ge1$, $N$ has zero ambient Riemannian volume. The transport identity for the locally Lipschitz weight $\rho_N^m w^\alpha$ thus holds by differentiating off $N$ and using dominated convergence on compact flow tubes. Applying the transport formula for the moving domain $D(t) = \varphi_t(D)$, we obtain
\begin{equation}
\label{eq:proof_main_rho}
\begin{array}{lll}
\frac{dV}{dt} = \int_{D(t)} \bigl( \mathcal{L}_f (\rho_N^m w^\alpha) + (\rho_N^m w^\alpha) \diverg f \bigr) dV.
\end{array}
\end{equation}
We compute the Lie derivative of the integrand in \eqref{eq:proof_main_rho}:
\begin{equation*}
\begin{array}{lll}
\mathcal{L}_f (\rho_N^m) = m \rho_N^{m-1} \mathcal{L}_f (\rho_N) = m \rho_N^{m-1} f^{\top} \nabla\rho_N,
\\
\mathcal{L}_f (w^\alpha) = \alpha w^{\alpha-1} \mathcal{L}_f (w) = \alpha w^\alpha f^{\top} \nabla\ln w.
\end{array}
\end{equation*}
Hence,
$\mathcal{L}_f (\rho_N^m w^\alpha) = \rho_N^m \,\alpha w^\alpha f^{\top} \nabla\ln w + m \rho_N^{m-1} w^\alpha f^{\top} \nabla\rho_N$.
Adding the term $(\rho_N^m w^\alpha) \diverg f$, we obtain the total rate of change of the integral:
$
\frac{dV}{dt} = \int_{D(t)} \rho_N^m w^\alpha \Bigl( \diverg f + \alpha f^{\top} \nabla\ln w + m \frac{f^{\top} \nabla\rho_N}{\rho_N} \Bigr) dV$.

Using the definition \eqref{eq:Dwalpha} of the weighted divergence of order $\alpha$, we get
\begin{equation*}
\frac{dV}{dt} = \int_{D(t)} \rho_N^m w^\alpha \Bigl( D_{w^\alpha}f + m \frac{f^{\top} \nabla\rho_N}{\rho_N} \Bigr) dV.
\end{equation*}

By the hypothesis of the theorem, in the positively invariant tubular neighbourhood $U$, inequality \eqref{eq:wcond_alpha} holds for $x \in U \setminus N$.
Since $D \subset U$ and $U$ is positively invariant, $D(t) \subset U$ for all $t \ge 0$. For almost every point of $D(t)$ (except possibly a set of measure zero where $\rho_N=0$), the estimate holds. Substituting it under the integral sign, we have
\begin{equation*}
\frac{dV}{dt} \le -\lambda \int_{D(t)} \rho_N^m w^\alpha \, dV = -\lambda V(t).
\end{equation*}

Thus $V(t)$ satisfies the differential inequality $\dot V(t) + \lambda V(t) \le 0$. Multiplying by the integrating factor $e^{\lambda t}$, we get
$\frac{d}{dt}\bigl( V(t) e^{\lambda t} \bigr) \le 0$.
Integrating from $0$ to $t$, we find $V(t) e^{\lambda t} \le V(0)$, or equivalently,
$\mathcal{V}_{N,w^\alpha}(\varphi_t(D)) = V(t) \le e^{-\lambda t} V(0) = e^{-\lambda t} \, \mathcal{V}_{N,w^\alpha}(D)$.
Theorem \ref{th:wNmain_alpha} is proved.
\end{proof}

\begin{example}\label{ex:Nw_cycle_corrected}
On $M=\mathbb R^2\setminus\{0\}$ consider
\[
\dot r=r(1-r),\qquad \dot\theta=4+3\sin\theta.
\]
The circle $N=\{r=1\}$ is an attracting limit cycle of codimension $m=1$. In the polar coordinate basis the correct Euclidean divergence is
\[
\operatorname{div}f=\frac1r\partial_r(r\dot r)+\partial_\theta\dot\theta
=2-3r+3\cos\theta.
\]
Choose the positive angular weight $w(\theta)=(4+3\sin\theta)^{-1}$. Since
$\mathcal L_f\ln w=-3\cos\theta$, the weighted divergence is
\[
D_{w^\alpha}f=2-3r+3(1-\alpha)\cos\theta.
\]
For $\rho_N=|r-1|$ and $r\ne1$ one has
$\mathcal L_f\rho_N/\rho_N=-r$. Hence at $\alpha=1$,
\[
D_wf+\frac{\mathcal L_f\rho_N}{\rho_N}=2-4r.
\]
The open annulus $U=\{0.85<r<1.15\}$ is a positively invariant tubular neighbourhood of $N$, and $2-4r\le-1.4$ on $U$. Theorem~\ref{th:wNmain_alpha} therefore gives the explicit distance-weighted ambient-volume estimate
\[
\mathcal V_{N,w}(\varphi_t(D))\le e^{-1.4t}\mathcal V_{N,w}(D),
\qquad D\Subset U.
\]
The actual pointwise radial attraction is independently apparent from $\dot r=r(1-r)$.
\end{example}

\section{Differential-geometric interpretation and conformal equivalence}
\label{sec:wforms}

For any $\alpha\in\mathbb{R}$, consider the conformal metric
\begin{equation}
\label{eq:conf_metric_alpha}
\tilde g_\alpha = w^{2\alpha/n} g,
\end{equation}
where $g$ denotes the original Riemannian metric on $M$.
Then $\tilde g_\alpha = e^{2\alpha\phi}g$, with $\phi=\frac{1}{n}\ln w$. 
The volume form is
\begin{equation}
\label{eq:form_metric}
\begin{array}{lll}
\omega_{\tilde g_\alpha} & =  \sqrt{\det \tilde g_\alpha}\, dx^1\wedge\dots\wedge dx^n 
\\
& = (w^{2\alpha/n})^{n/2} \sqrt{\det g}\, dx^1\wedge\dots\wedge dx^n = w^\alpha \omega_g.
\end{array}
\end{equation}
Thus, the weighted volume form in $g$ is precisely the ordinary volume form in $\tilde g_\alpha$. The divergence of $f$ in the metric $\tilde g_\alpha$ is given by
\begin{equation}
\label{eq:div_new_metric}
\begin{array}{lll}
\widetilde{\diverg}_\alpha f & = \diverg_g f + \alpha n f^{\top} \nabla \phi  
\\
& = \diverg_g f + \alpha f^{\top} \nabla \ln w = D_{w^\alpha}f,
\end{array}
\end{equation}
where $\diverg_g f$ measures the rate of change of the Riemannian volume (in the metric $g$) under the flow of $f$.
Thus $D_{w^\alpha}f$ is the ordinary divergence in the metric $\tilde g_\alpha$. The family of metrics $\{\tilde g_\alpha\}_{\alpha\in\mathbb{R}}$ is generated by the conformal rescalings $T_\alpha(g)=w^{2\alpha/n}g$, with $T_\alpha\circ T_\beta=T_{\alpha+\beta}$ for the fixed function $w$. The sign of $\alpha$ does not by itself determine contraction or expansion.

We now state the geometric equivalence theorem for any $\alpha$.

\begin{theorem}
\label{th:geom_equiv_alpha}
The following statements hold:
\begin{enumerate}
    \item[(a)] $\Omega_{w^\alpha}(D) = \operatorname{vol}_{\tilde g_\alpha}(D)$.
    \item[(b)] The system is $w^\alpha$-volumetrically conservative in $g$ $\Leftrightarrow$ $\widetilde{\diverg}_\alpha f\equiv0$.
    \item[(c)] The system is $w^\alpha$-volumetrically dissipative in $g$ $\Leftrightarrow$ $\widetilde{\diverg}_\alpha f\le 0$.
    \item[(d)] Asymptotic $w^\alpha$-contractivity in $g$ is equivalent to asymptotic volumetric contractivity in $\tilde g_\alpha$.
\end{enumerate}
\end{theorem}

\begin{proof}
Consider the conformal metric \eqref{eq:conf_metric_alpha}. 
Its volume form is \eqref{eq:form_metric}.
Hence, for any measurable domain $D$,
$\operatorname{vol}_{\tilde g_\alpha}(D) = \int_D \omega_{\tilde g_\alpha} = \int_D w^\alpha \omega_g = \Omega_{w^\alpha}(D)$,
which proves (a).

The divergence of $f$ in $\tilde g_\alpha$ is expressed via the divergence in $g$ and the conformal factor as in \eqref{eq:div_new_metric}.
From \eqref{eq:Dwalpha} we have $\widetilde{\diverg}_\alpha f = D_{w^\alpha}f$.

Using \eqref{eq:deriv_alpha}, the rate of change of the weighted volume is
$\frac{d}{dt}\Omega_{w^\alpha}(D(t)) = \int_{D(t)} D_{w^\alpha}f \; w^\alpha\, dV_g$.
But $w^\alpha\, dV_g = \omega_{\tilde g_\alpha}$ and $D_{w^\alpha}f = \widetilde{\diverg}_\alpha f$, so
$\frac{d}{dt}\Omega_{w^\alpha}(D(t)) = \int_{D(t)} \widetilde{\diverg}_\alpha f \; \omega_{\tilde g_\alpha}$.
Exactly the same formula holds for the ordinary Riemannian volume in $\tilde g_\alpha$:
\begin{equation*}
\frac{d}{dt}\operatorname{vol}_{\tilde g_\alpha}(D(t)) = \int_{D(t)} \widetilde{\diverg}_\alpha f \; \omega_{\tilde g_\alpha} .
\end{equation*}
Thus the functions $\Omega_{w^\alpha}(D(t))$ and $\operatorname{vol}_{\tilde g_\alpha}(D(t))$ have the same derivative and coincide at $t=0$. Hence they are equal for all $t\ge0$.

We now prove the equivalences in (b)--(d).

The system is $w^\alpha$-volumetrically conservative in $g$ iff for every $D\Subset M$, $\frac{d}{dt}\Omega_{w^\alpha}(D(t))\equiv0$. Since $w^\alpha>0$ and $D$ is arbitrary, this is equivalent to $D_{w^\alpha}f\equiv0$. As $D_{w^\alpha}f = \widetilde{\diverg}_\alpha f$, this is equivalent to $\widetilde{\diverg}_\alpha f\equiv0$, i.e. volumetric conservativity in $\tilde g_\alpha$.

Similarly, $w^\alpha$-volumetric dissipativity means $\Omega_{w^\alpha}(\varphi_t(D))\le \Omega_{w^\alpha}(D)$ for all $D$ and $t\ge0$. In the smooth case this is equivalent to $D_{w^\alpha}f\le0$, which, since $D_{w^\alpha}f = \widetilde{\diverg}_\alpha f$, becomes $\widetilde{\diverg}_\alpha f\le0$, exactly the dissipativity condition in $\tilde g_\alpha$.

Asymptotic $w^\alpha$-contractivity in $g$ means that the weighted volume of $\varphi_t(D)$ tends to zero for every $D\Subset M$. Since weighted volume in $g$ equals ordinary volume in $\tilde g_\alpha$ at each time, the same convergence holds in the conformal metric. The equivalence is established.
Theorem \ref{th:geom_equiv_alpha} is proved.
\end{proof}

\begin{remark}
\label{sec:family} 
Consider the family of $w^\alpha$-volumes and their duality.
The introduced family $\Omega_{w^\alpha}$ for different $\alpha$ possesses a certain symmetry. From \eqref{eq:Dwalpha} we have
\begin{equation}
\label{eq:duality}
D_{w^\alpha}f + D_{w^{-\alpha}}f = 2\diverg f .
\end{equation}
In particular, if $\diverg f \equiv 0$ (the system preserves ordinary volume), then $D_{w^{-\alpha}}f = - D_{w^\alpha}f$. 
If ordinary divergence vanishes and one weighted divergence is uniformly negative on a common invariant region, the opposite weight has uniformly positive divergence there.
More generally, for any $\alpha,\beta\in\mathbb{R}$, $\beta\neq0$, we have
\begin{equation*}
D_{w^\alpha}f = \frac{\alpha}{\beta} D_{w^\beta}f + \left(1-\frac{\alpha}{\beta}\right)\diverg f,
\end{equation*}
which expresses an affine dependence on $\alpha$ (interpolation only when $\alpha/\beta\in[0,1]$). This identity permits checking the desired sign for different $\alpha$. The existence of a suitable exponent requires the relevant pointwise inequalities to have a common solution.
\end{remark}

\begin{remark}
The identity for inverse weights must not be interpreted as a change in dynamical attraction or repulsion.
Changing $\alpha$ changes the reference measure but does not change the vector field, invariant sets, or their dynamical attraction and repulsion. Equation~\eqref{eq:duality} only relates the corresponding local weighted divergences. Without information about $\operatorname{div}f$, it does not imply opposite volume behaviour.
\end{remark}

\begin{example}
Return to the system with a limit cycle \eqref{ex:sys_first} and the weight $w=e^{16(r-1)^2}$. For the given neighbourhood $U_{1/2}=\{|r-1| \leq \frac{1}{2}\}$, the expression $D_{w^\alpha}f = -1-3u - 32\alpha(1+u)u^2$ (where $u=r-1$ and $u \in [-0.5,0.5]$) shows that increasing $\alpha$ lowers the weighted divergence at every $r\ne1$, whereas decreasing it raises the weighted divergence there. These statements concern the chosen weighted-volume estimate and do not alter dynamical stability.
\end{example}

\begin{remark}
In some thermostatted systems \cite{Nose,Hoover,EvansMorriss}, invariant or stationary densities can arise under additional assumptions. 
The proposed criteria expressed in terms of $D_{w^\alpha}f$ provide a rigorous tool for their stability analysis, and different exponents define different reference measures, but need not correspond to physical invariant ensembles.
\end{remark}

\section{Oriented volumetric stability with respect to an invariant hypersurface}
\label{sec:orient}

In the previous sections, the positive weight function $w^\alpha$ and the parameter $\alpha$ were chosen freely in order to ``tune'' the measure with respect to which the contraction of the phase flow is studied.  
Now we consider the case where the geometry of the problem itself gives rise to a natural sign-alternating weight function.  
Let $\Sigma=\{\psi=0\}$ be a smooth hypersurface invariant under the flow of $f$. Then on $\Sigma$ we have $f^\top\nabla\psi=0$, and the scalar product $w(x)=f(x)^\top\nabla\psi(x)$ vanishes on $\Sigma$ and, as a rule, changes sign when crossing it, indicating the direction of the phase flow relative to the hypersurface.  
The oriented volume $V_\psi(D)=\int_D w(x)\,dV(x)$ is a signed volume integral of the normal defining-function derivative (not itself a surface-flux integral) and is a natural analogue of the weighted volume, in which the conformal weight is replaced by a function generated by the dynamics itself.  
Unlike the arbitrary choice of $w$ and $\alpha$, here the weight is fixed by the invariant hypersurface, and the parameter $\alpha$ is not introduced. The evolution of the oriented volume is described by the divergence of the modified field $wf$, and under additional conditions the nonnegative functional $\int_D|w|\,dV$ decreases and $|w|$ also bounds the trajectory-wise distance to $\Sigma$.  
This approach combines a pointwise Lyapunov-type estimate for $|w|$ with an integral volume estimate near an invariant hypersurface.

Consider a smooth system on an open set $O\subseteq\mathbb R^n$ and a smooth compact hypersurface $\Sigma = \{x\in O : \psi(x)=0\}$, where $\psi$ is smooth on $O$, $\nabla\psi\neq0$ on $\Sigma$, and $\Sigma$ is invariant under the flow: $f^{\top}\nabla\psi=0$ on $\Sigma$. In this section the weight function may change sign and need not be a power of a positive function, so the parameter $\alpha$ is not used.

Define the sign-alternating weight function
\begin{equation*}
w(x) := f(x)^{\top} \nabla\psi(x).
\end{equation*}
On $\Sigma$ we have $w=0$. Introduce the oriented $\psi$-volume of a domain $D$:
\begin{equation}\label{eq:Vpsi}
V_\psi(D) = \int_D w(x)\,dV(x) = \int_D f(x)^{\top} \nabla\psi(x) \,dV(x) .
\end{equation}

We now state the conditions for oriented volumetric stability.

\begin{theorem}\label{th:main_orient}
Let $O\subseteq\mathbb R^n$ be open and let $f\in C^1(O;\mathbb R^n)$ generate a forward-complete flow on $O$. Let $\psi\in C^2(O)$ and let $\Sigma=\{x\in O:\psi(x)=0\}$ be a smooth compact invariant hypersurface, with $\nabla\psi\neq0$ on $\Sigma$. Set $w=\mathcal L_f\psi$. Suppose there is a tubular neighbourhood $U\Subset O$ of $\Sigma$ and constants $c_1,c_2,\lambda,\mu>0$, with $\mu<\lambda$, such that
\begin{enumerate}
\item[(a)] $w^{-1}(0)\cap U=\Sigma$, $\nabla w$ is transverse to $\Sigma$, and
$c_1\operatorname{dist}(x,\Sigma)\le|w(x)|\le c_2\operatorname{dist}(x,\Sigma)$ on $U$;
\item[(b)] $\operatorname{sgn}(w)\operatorname{div}(wf)\le-\lambda|w|$ on $U\setminus\Sigma$;
\item[(c)] $\lambda+\operatorname{div}f\ge\mu$ on $U$.
\end{enumerate}
Then $\Sigma$ is locally exponentially asymptotically stable. More precisely, there exists a positively invariant neighbourhood $U_0\Subset U$ of $\Sigma$ such that every $x_0\in U_0$ satisfies
\[
\operatorname{dist}(\varphi_t(x_0),\Sigma)
\le \frac{c_2}{c_1}e^{-\mu t}\operatorname{dist}(x_0,\Sigma).
\]
Moreover, for every measurable $D\Subset U_0$,
\[
\int_{\varphi_t(D)}|w|\,dV
\le e^{-\lambda t}\int_D|w|\,dV,\qquad t\ge0.
\]
\end{theorem}
\begin{proof}
For a trajectory in $U$, let $v(t)=|w(\varphi_t(x_0))|$. On $U\setminus\Sigma$, the identity
$\mathcal L_fw=\operatorname{div}(wf)-w\operatorname{div}f$
and assumptions (b),(c) give
\[
\dot v\le-\bigl(\lambda+\operatorname{div}f\bigr)v\le-\mu v.
\]
On $\Sigma$, $w=0$ and invariance ensures that $v$ remains zero, so the same inequality holds along all trajectories in $U$. Since $\Sigma$ is compact, choose $h>0$ so that the closed $h$-tube about $\Sigma$ is contained in $U$. Fix $0<\delta<c_1h$ and put
$U_0:=\{x\in O:\dist(x,\Sigma)<h,\ |w(x)|<\delta\}$.
By (a), its closure lies in $U$, because $|w|\le\delta$ implies $\dist(x,\Sigma)\le\delta/c_1<h$ within the $h$-tube. The inequality $v(t)\le e^{-\mu t}v(0)$ holds up to any possible first exit from $U_0$. Such an exit cannot occur through $|w|=\delta$ (since $v$ decreases) or through $\dist(x,\Sigma)=h$ (where (a) gives $|w|\ge c_1h>\delta$). Thus $U_0$ is positively invariant, and the two bounds in (a) imply the distance estimate.

For $D\Subset U_0$, put $W(t)=\int_{\varphi_t(D)}|w|\,dV$.
The transport formula applies to the locally Lipschitz weight $|w|$. Since $w^{-1}(0)\cap U=\Sigma$ has zero ambient volume,
\begin{align*}
\dot W(t)
=\int_{\varphi_t(D)}\operatorname{sgn}(w)
     \operatorname{div}(wf)\,dV
\le-\lambda W(t).
\end{align*}
Gronwall's inequality completes the proof.
\end{proof}

\begin{example}
Consider system \eqref{ex:sys_first} on the open set $O=\mathbb R^2\setminus\{0\}$.
The hypersurface $\Sigma=\{r=1\}$ is invariant. Setting $\psi(x,y)=r-1$, we find $w = f^{\top}\nabla\psi = r(1-r)$. Condition (a) of Theorem \ref{th:main_orient} is satisfied. Compute $wf = (r^2(1-r)^2, r(1-r))^{\top}$. In polar coordinates,
$\diverg(wf) = \frac{1}{r}\frac{d}{dr} r^3(1-r)^2 = r(1-r)(3-5r)$.
Then $\sgn(w)\diverg(wf) = (3-5r)|w|$. 
In the open tubular neighbourhood $U=\{0.9<r<1.1\}$ we have $3-5r \le -1.5$, so condition (b) of Theorem \ref{th:main_orient} holds with $\lambda=1.5$. Furthermore, the minimum of $\diverg f = 2-3r$ on the interval is at $r=1.1$, giving $-1.3$. 
Choose $\mu=0.2$, then $\lambda+\diverg f \ge 0.2 >0$. By the theorem, the distance to $\Sigma$ decays at least as $e^{-0.2t}$ for initial points in a sufficiently small $U_0\Subset U$, and the nonnegative integral of $|w|$ decays at rate $1.5$ for domains $D\Subset U_0$. This does not assert monotonicity of the signed integral $V_\psi(D)$.
\end{example}

\section{Conclusion}

In this paper, we have introduced and studied a one-parameter family of weighted phase volumes $\Omega_{w^\alpha}$ with weight $w^\alpha$, $\alpha\in\mathbb{R}$. This family includes the ordinary phase volume at $\alpha=0$ and the mutually inverse weights at $\alpha=\pm1$, without changing the underlying dynamical system. Sufficient conditions for weighted-volume contraction and expansion and for excluding invariant Jordan sets of positive measure have been proved for any fixed $\alpha$. The zero-volume attractor result additionally assumes a compact absorbing set. A geometric interpretation has been given, showing that $D_{w^\alpha}f$ is the divergence in the conformal metric $\tilde g_\alpha = w^{2\alpha/n}g$. The conformal rescalings form a one-parameter group for fixed $w$, and the identity $D_{w^\alpha}f+D_{w^{-\alpha}}f=2\operatorname{div}f$ relates the two weighted divergences without implying a change in the attractor--repeller structure. The signed-weight argument additionally combines a pointwise Lyapunov-type estimate with a volume estimate near an invariant hypersurface.

\bibliographystyle{unsrtnat}

\bibliography{weighted_volume_refs_corrected_v3}

\end{document}